\documentclass[12pt]{article}
\usepackage{amsfonts}
\usepackage{amssymb}
\usepackage{mathrsfs}
\usepackage{srcltx}
\usepackage{amsmath}
\usepackage{amsthm}
\usepackage{amstext}
\usepackage{amsopn}
\usepackage{graphicx}
\usepackage{url}
\usepackage{color}

\newtheorem{theorem}{Theorem}[section]

\newtheorem{corollary}[theorem]{Corollary}

\theoremstyle{definition}

\newtheorem{example}[theorem]{Example}

\theoremstyle{remark}

\numberwithin{equation}{section}

\newcommand{\ba}{\begin{array}}
\newcommand{\ea}{\end{array}}

\begin{document}
\date{}
\title{\bf\Large{From Sylvester's determinant identity to Cramer's rule}\footnote{Partially supported by a grant from China Scholarship Council and National Natural Science Foundation of China (11101071, 11271001, 51175443).}}
\author{Hou-biao Li\textsuperscript{1}\footnote{Corresponding Author, Email: lihoubiao0189@163.com}, Ting-Zhu Huang\textsuperscript{1}, Tong-xiang Gu\textsuperscript{2}, Xing-Ping Liu\textsuperscript{2}
 \\
{\small \textsuperscript{1} School of Mathematical Sciences, University of Electronic Science}\\
{\small and Technology of China, Chengdu, 611731, P.R. China.}\\
{\small \textsuperscript{2} Lab of Comp. Phy., Institute of Applied Physics and Computational Mathematics,}\\
\ \ {\small Beijing, 100088, P. R. China.}}

\maketitle

\begin{abstract}
{The object of this paper is to introduce a new and fascinating method of solving large
linear equations, based on Cramer's rule or Gaussian elimination but employing Sylvester's determinant identity in its computation process. In addition, a scheme
suitable for parallel computing is presented for this kind of generalized Chi\`{o}'s determinant condensation processes,
which makes this new method have a property of natural parallelism. Finally, some numerical
experiments also confirm our theoretical analysis.}

\end{abstract}

\section {Introduction}
As is well-known, how to solve effectively linear systems
is a very important problem in scientific and engineering fields. Many of linear solvers have been researched such as Gaussian elimination \cite{Gaussian,Saad}, relaxation methods \cite{Young71}, row-action iteration schemes
\cite{Row92,Householder60} and (block) Krylov subspace \cite{Bloch2011,Saad}.

Recently, a low communication condensation-based linear system solver utilizing Cramer's Rule is presented in \cite{ken2012}.
As the authors stated that unique combination between Cramer's rule and matrix condensation techniques yields an elegant parallel computing architectures, by constructing a binary, tree-based data flow in which the algorithm mirrors the matrix at critical points during the condensation process.
Moreover, the accuracy and computational complexity of the proposed algorithm are similar to LU-decomposition \cite{Gaussian}.

In this paper, we will continue research this kind of parallel algorithms and give some theoretical analysis and a generalized Chi\`{o}'s determinant condensation process, which
perfect the corresponding conclusions.

This paper is organized as follows. In Second 2, we will review some more general determinant condensation algorithms---Sylvester's determinant identity,
and then give theoretical basis on the above parallel computing architectures \cite{ken2012}, which shows the negation in mirroring process is not
necessary to arrive at the correct answer. Moreover, a more general scheme utilizing Cramer's Rule and matrix condensation techniques is also given.
In addition, the scheme suitable for parallel computing on the sylvester's identity is proposed in Section 3. Finally, a simple example is used to illustrate this new algorithm in Section 4.


\section{Sylvester's determinant condensation algorithms}

Throughout this section, we mainly consider an $n\times n$ matrix $A=(a_{ij})$ $(i, j = 1,2,\ldots, n)$ with elements $a_{ij}$ and
determinant $|A|$, also written $\det A$. Recently, a Chi\`{o} condensation method \cite{1853} is applied to solve large linear systems in \cite{ken2012}.
In fact, the prototype of this method may be traced back to the following Sylvester's determinant identity for
calculating a determinant of arbitrary order in 1851.

\begin{theorem}\label{Th2.1}(Sylvester's identity,\cite{MR2419937,LAA2014,MCS1996,AM2005}).
Let $A=(a_{ij})$ be an $n\times n$ matrix over a commutative ring. For a submatrix $A_0=(a_{ij})$, $i,j=1,\ldots k$ of $A$, set
\begin{equation}\label{2.1}
{\hat c_{pq}} = \det \left[ {\begin{array}{*{20}{c}}
{}&{}&{}&{{a_{1q}}}\\
{}&{{A_0}}&{}& \vdots \\
{}&{}&{}&{{a_{kq}}}\\
{{a_{p1}}}& \cdots &{{a_{pk}}}&{{a_{pq}}}
\end{array}} \right].
\end{equation}
Let $\hat C = ({{\hat c}_{pq}})$, $p, q=k+1,\ldots,n$. Then
$$
{{(\det {A_0})}^{n - k - 1}}\det A = {\det \hat C}.
$$
Specially when $A_0$ is an invertible matrix, we have that
\begin{equation}\label{2.2}
\det A = \frac{{\det \hat C}}{{{{(\det {A_0})}^{n - k - 1}}}}.
\end{equation}
\end{theorem}

\begin{corollary}\label{Co2.2}(Chi\`{o}'s method, \cite{F75,ken2012}).
For an $n\times n$ matrix $A=(a_{ij})$ with $a_{nn}\neq 0$, let $E =(e_{ij})$ be the $(n-1)\times(n-1)$ matrix defined by
\begin{equation}\label{2.3}
{e_{ij}} = \left| {\begin{array}{*{20}{c}}
{{a_{ij}}}&{{a_{in}}}\\
{{a_{nj}}}&{{a_{nn}}}
\end{array}} \right| = {a_{ij}}{a_{nn}} - {a_{in}}{a_{nj}},{\kern 1pt} {\kern 1pt} {\kern 1pt} {\kern 1pt} i,j = 1, \ldots ,n - 1.
\end{equation}
Then
\begin{equation}\label{2.4}
\det A = \frac{1}{{a_{nn}^{n - 2}}}\det E.
\end{equation}
\end{corollary}

Obviously, the above Theorem \ref{Th2.1} reduces a matrix of order $n$ to order $n-k$ to evaluate its determinant. Repeating the
procedure numerous times can reduce a large matrix to a small one, which is convenient for the calculation. This process is called by condensation method \cite{1853,F75}.
As an example of Chi\`{o}'s condensation, the paper \cite{ken2012} considers the following $3\times 3$ matrix:
\[A = \left| {\begin{array}{*{20}{c}}
{{a_{11}}}&{{a_{12}}}&{{a_{13}}}\\
{{a_{21}}}&{{a_{22}}}&{{a_{23}}}\\
{{a_{31}}}&{{a_{32}}}&{{a_{33}}}
\end{array}} \right|{\kern 1pt} {\kern 1pt} {\kern 1pt} {\rm{and}}{\kern 1pt} {\kern 1pt} {\rm{its}}{\kern 1pt} {\kern 1pt} {\rm{condensed}}{\kern 1pt} {\kern 1pt} {\rm{form}}{\kern 1pt} {\kern 1pt} \left| {\begin{array}{*{20}{c}}
\times&\times&\times\\
\times&{{a_{11}}{a_{22}} - {a_{21}}{a_{12}}}&{{a_{11}}{a_{23}} - {a_{21}}{a_{13}}}\\
\times&{{a_{11}}{a_{32}} - {a_{31}}{a_{12}}}&{{a_{11}}{a_{33}} - {a_{31}}{a_{13}}}
\end{array}} \right|.\]

In fact, the above condensation processes are not only used to evaluate determinants but also can be used to solve linear systems.
For example, one can derive the following equivalence relation on the solution formula of linear systems.

\begin{theorem}\label{Th2.2}(Equivalence relation). The linear system in the form
$Ax=b$ (where $A=(a_{ij})$ is an $n\times n$ invertible coefficient matrix) has the same corresponding solution as the linear system $\hat C{x^{(k)}} = {b^{(k)}}$, where
${\hat C}$ is defined as in Theorem \ref{Th2.1}, ${x^{(k)}} = {[{x_{k + 1}},{x_{k + 2}}, \ldots ,{x_n}]^T}$ and
${b^{(k)}} = {[b_{k + 1}^{'},b_{k + 2}^{'}, \ldots,b_n^{'}]^T}$. Here
\[b_j^{'} = \det \left[ {\begin{array}{*{20}{c}}
{}&{}&{}&{{b_1}}\\
{}&{{A_0}}&{}& \vdots \\
{}&{}&{}&{{b_k}}\\
{{a_{j1}}}& \cdots &{{a_{jk}}}&{{b_j}}
\end{array}} \right],j = k + 1, \ldots ,n.\]
\end{theorem}

\begin{proof} According to Theorem \ref{Th2.1} or Eq. \eqref{2.2}, we know that there exists a constant ${{(\det {A_0})}^{n - k - 1}}$ between the determinant of $A$ and the determinant of $\hat C$,
which is only dependent on the given submatrix $A_0$. Therefore, for any given submatrix $A_0$, there also exists the same constant ${{(\det {A_0})}^{n - k - 1}}$ between the determinant of
$A_j(b)$ ($j = k + 1, \ldots ,n$), the matrix $A$ with its $jth$ column replaced
by $b$, and the determinant of $\hat C_j(b^{(k)})$. Thus, by Cramer's rule, we have that
\[{x_j} = \frac{{\det ({A_j}(b))}}{{\det A}} = \frac{{det{{({A_0})}^{n - k - 1}}\det ({{\hat C}_j}({b^{(k)}}))}}{{det{{({A_0})}^{n - k - 1}}\det \hat C}} = \frac{{\det ({{\hat C}_j}({b^{(k)}}))}}{{\det \hat C}},j = k + 1, \ldots ,n.\]
The conclusion holds.
\end{proof}

Obviously, when the submatrix $A_0$ is singular, the solution of linear systems cannot be evaluated by this method. Since interchanging the $r$th and $n$th rows
and the $s$th and $n$th columns of linear systems has only an effect on the order of the unknowns $x_i$, which has no effect on the whole solution $x$.
Therefore, we may obtain the following more general conclusion.

For convenience, we firstly define the ordered index list $N_n=(1,2,\ldots,n)$ for any positive integer $n$.
For two ordered index (i.e., for any $\alpha < \beta, i_{\alpha}<i_{\beta}$) lists $I=(i_1,\ldots,i_t)\subset N_n$ and $J=(j_1,\ldots,j_t)\subset N_n$, we denote the corresponding
complementary ordered index lists by $I'$ and $J'$, respectively. That is, $I\bigcup I'=J\bigcup J'=N_n$.

\begin{corollary}\label{Th3.4}.
Let $A=(a_{ij})$ be an $n\times n$ matrix and $k$ be a fixed integer $0\leq k \leq n-1$.
$I=(i_1,\ldots,i_k)\subset N_n$ and $J=(j_1,\ldots,j_k)\subset N_n$ are two ordered index lists. We denote the corresponding
submatrix, extracted from $A$, as
\[A\left[ {\begin{array}{*{20}{c}}
I\\
J
\end{array}} \right] = A\left[ {\begin{array}{*{20}{c}}
{{i_1}, \ldots ,{i_k}}\\
{{j_1}, \ldots ,{j_k}}
\end{array}} \right] \triangleq \left[ {\begin{array}{*{20}{c}}
{{a_{{i_1}{j_1}}}}& \cdots &{{a_{{i_1}{j_k}}}}\\
 \vdots & \ddots & \vdots \\
{{a_{{i_k}{j_1}}}}& \cdots &{{a_{{i_k}{j_k}}}}
\end{array}} \right].\]
Suppose that the invertible submatrix $A_0=A\left[ {\begin{array}{*{20}{c}}
I\\
J
\end{array}} \right]$ in the Theorem \ref{Th2.1}, then the linear system $Ax=b$ has the same corresponding solutions
as the linear system $\hat C{x^{(k)}_{J'}} = {b^{(k)}_{J'}}$, where
${\Box_{J'}}$ is defined as the subset of $\Box$ with the index coming from $J'$.
\end{corollary}

According to the above theorem, one can easily see that though the condensation process removes information associated with discarded columns, we may obtain
certain variables values by controlling the elements in the set $J'$, see Example \ref{E4.1}.

In addition, the matrix mirroring and the negation of matrix mirroring process in \cite{ken2012} are also not necessary to arrive at the correct answer,
since we may obtain the similar parallel computing process by condensing the index set $J'$ from both sides (left and right), see Figure 1.
\[\left| {\begin{array}{*{20}{c}}
{{a_{11}}}&{{a_{12}}}&{{a_{13}}}\\
{{a_{21}}}&{{a_{22}}}&{{a_{23}}}\\
{{a_{31}}}&{{a_{32}}}&{{a_{33}}}
\end{array}} \right| \to {\mathop{\rm mirrored}\nolimits}  \to \left| {\begin{array}{*{20}{c}}
{{a_{13}}}&{{a_{12}}}&{ - {a_{11}}}\\
{{a_{23}}}&{{a_{22}}}&{ - {a_{21}}}\\
{{a_{33}}}&{{a_{32}}}&{ - {a_{31}}}
\end{array}} \right|.\]

\begin{figure}[htbp]\label{fig1}
\centerline{\includegraphics[width=6.20in,height=4.7in]{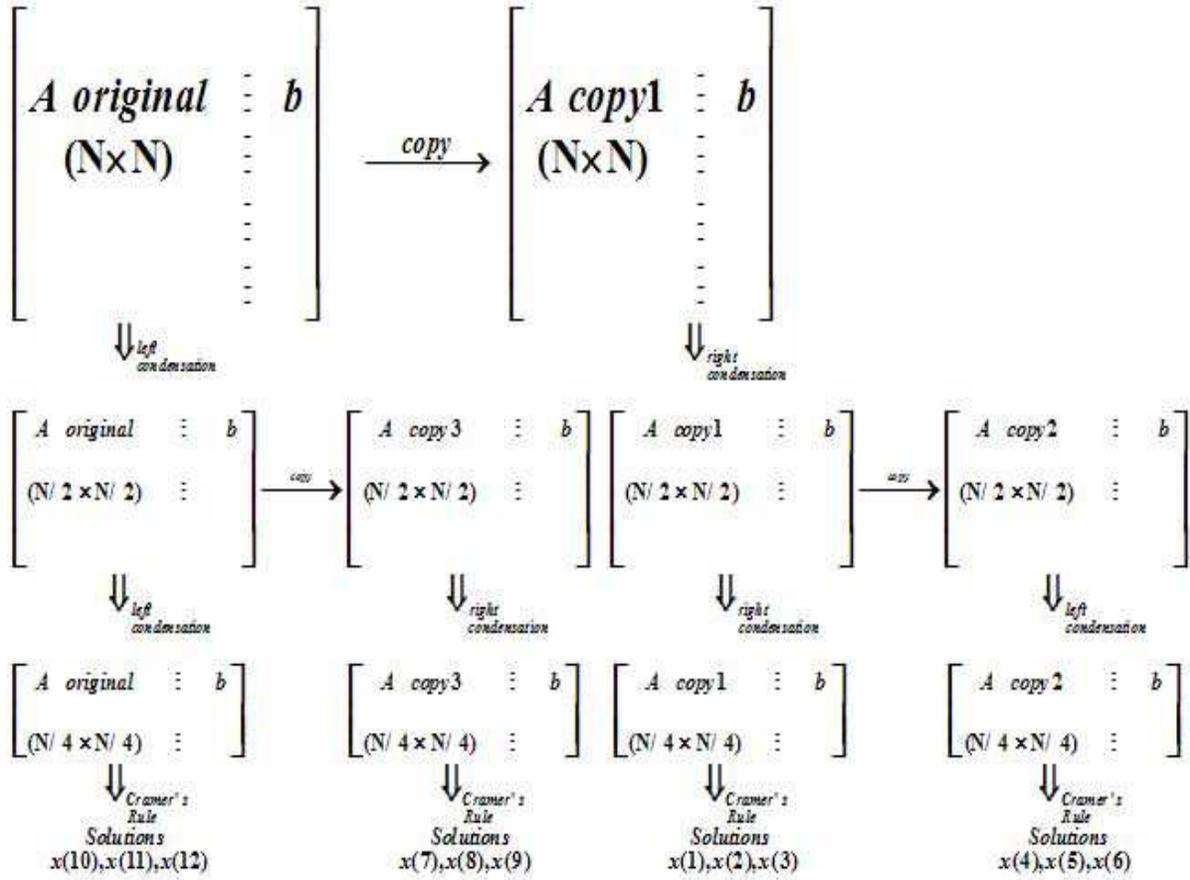}}
\caption{\small A process flow depicting the proposed framework.}
\end{figure}

Similar to \cite{ken2012}, copying occurs with the initial matrix and then each time a matrix is reduced
in half. An $N \times N$ matrix is copied when it reaches the size of $N\setminus 2 \times N\setminus 2$. Once the
matrix is copied, there is double the work. In other words, two $N\setminus 2 \times N\setminus 2$ matrices each
require a condensation. Obviously, the amount of
work for two matrices of half the size is much lower than that of one $N\times N$ matrix,
which avoids the $O(N^4)$ growth pattern in computations. This is due to the $O(N^3)$
nature of the condensation process (see \cite{ken2012}).

Similarly, one may consider a scenario
in which the algorithm creates more than two matrices during each copying step, according to Corollary \ref{Th3.4}.
On its computational complexity and more details, see \cite{ken2012}.

\section {A scheme suitable for parallel computing on the Sylvester's identity}

The Sylvester's identity \ref{Th2.1} reduces a matrix of order $n$ to order $n-k$ when evaluating its determinant.
Since when $k=1$, it is just the Chi\`{o}'s method. Therefore, for convenience, we call the Sylvester's identity a {\bf{K-Chi\`{o}'s method}} from now on.

As have been shown above, repeating the procedure numerous times can reduce a large matrix to a size convenient for the computations.
However, in order to condense a matrix from $N\times N$ to $(N-M)\times(N-M)$, the core
calculation is repeated $(N-M)^2$ times. Obviously, this is very expensive. In fact, we may parallel computing
each row of the matrix ${\hat C}$ in \eqref{2.1}, since $A_0$ or $A\left[ {\begin{array}{*{20}{c}}
{{i_1}, \ldots ,{i_t}}\\
{{j_1}, \ldots ,{j_t}}
\end{array}} \right]$ is common to each element of the row and may be calculated but once for each row via expanding by the last column.
For example, for the $p$-th row $\alpha _p$ of $\hat C$, we may write
\begin{equation}\label{3.30}
\begin{array}{l}
{\alpha _p} = \left[ {{{\hat c}_{p,k + 1}},{{\hat c}_{p,k + 2}}, \ldots ,{{\hat c}_{p,k + n}}} \right]\\
 = \left[ {\underbrace {A_0^{p,1},A_0^{p,2}, \cdots ,A_0^{p,k},|{A_0}|}_{k + 1}} \right] \cdot \left[ {\begin{array}{*{20}{c}}
{{a_{1,k + 1}}}&{{a_{1,k + 2}}}& \cdots &{{a_{1,n}}}\\
{{a_{2,k + 1}}}&{{a_{2,k + 2}}}& \cdots &{{a_{2,n}}}\\
 \cdots & \cdots & \cdots & \cdots \\
{{a_{p,k + 1}}}&{{a_{p,k + 2}}}& \cdots &{{a_{p,n}}}
\end{array}} \right],
\end{array}
\end{equation}
where
\[A_0^{p,\mathbf{j}} =  - \left| {\begin{array}{*{20}{c}}
{{a_{11}}}&{{a_{12}}}& \cdots &{{a_{1k}}}\\
\begin{array}{l}
{a_{21}}\\
 \cdots
\end{array}&\begin{array}{l}
{a_{22}}\\
 \cdots
\end{array}&\begin{array}{l}
 \cdots \\
 \cdots
\end{array}&\begin{array}{l}
{a_{2k}}\\
 \cdots
\end{array}\\
\begin{array}{l}
{{ a_{\textbf p1}}}\\
 \cdots
\end{array}&\begin{array}{l}
{{ a_{\textbf p2}}}\\
 \cdots
\end{array}&\begin{array}{l}
 \cdots \\
 \cdots
\end{array}&\begin{array}{l}
{{ a_{\textbf pk}}}\\
 \cdots
\end{array}\\
{{a_{k1}}}&{{a_{k2}}}& \cdots &{{a_{kk}}}
\end{array}} \right|\begin{array}{*{20}{c}}
{}\\
{}\\
{}\\
\mathbf{j}\\
{}\\
{}
\end{array}.\]
Therefore, only $k$ determents $A_0^{p,\mathbf{j}}$ ($j=1,\ldots,k$) and a common $k\times k$ determent $\det A_0$ are needed
for each row of the matrix ${\hat C}$. Therefore, the matrix $\hat C$ is essentially suitable for parallel computations since the each row of matrix $\hat C$ may be
independently computed by \eqref{3.30}. See Example \ref{E3.3} below.

\begin{example}\label{E3.3} Consider the following four order determent
\[\left| A \right| = \left| {\begin{array}{*{20}{c}}
1&{ - 2}&3&1\\
4&2&{ - 1}&0\\
\mathbf{0}&\mathbf{2}&{1}&{5}\\
\mathbf{-3}&\mathbf{3}&1&2
\end{array}} \right|.\]

Let ${A_0} = \left[ {\begin{array}{*{20}{c}}
1&{ - 2}\\
4&2
\end{array}} \right]$, then $|A_0|=\mathbf{10}$, and
\begin{equation*}
\begin{array}{l}
{\alpha _3} = \left[ -{\left| {\begin{array}{*{20}{c}}
\mathbf{0}&\mathbf{2}\\
4&2
\end{array}} \right|,-\left| {\begin{array}{*{20}{c}}
1&-2\\
\mathbf{0}&\mathbf{2}
\end{array}} \right|,\mathbf{10}} \right]\left[ {\begin{array}{*{20}{c}}
3&1\\
{ - 1}&0\\
1&5
\end{array}} \right] = \left[ {\begin{array}{*{20}{c}}
{36}&{58}
\end{array}} \right];\\
{\alpha _4} = \left[ {-\left| {\begin{array}{*{20}{c}}
\mathbf{-3}&\mathbf{3}\\
4&2
\end{array}} \right|,-\left| {\begin{array}{*{20}{c}}
1&{ - 2}\\
\mathbf{-3}&\mathbf{3}
\end{array}} \right|,\mathbf{10}} \right]\left[ {\begin{array}{*{20}{c}}
3&1\\
{ - 1}&0\\
1&2
\end{array}} \right] = \left[ {\begin{array}{*{20}{c}}
{61}&{38}
\end{array}} \right].
\end{array}
\end{equation*}

Therefore,
\[\left| A \right| = \frac{1}{{{{\mathbf{10}}^{4 - 3}}}}\left| {\begin{array}{*{20}{c}}
{36}&{58}\\
{61}&{38}
\end{array}} \right| =  - 217.\]
\end{example}

From here, we note that only six $2\times 2$ determinants is needed. However, Chi\`{o}'s method will require fourteen $2\times 2$ determinants to be computed.
In addition, comparing with the Gaussian elimination, our method increases only two multiplications. But Gaussian elimination method is not too suitable for parallel computing. Thus, the whole computational amount on the matrix ${\hat C}$ will be much less than that
involved in the old process of computation \cite{ken2012}. Concretely speaking, if we denote the total of multiplications/diversions on the $k$-order determinant $|A_0|$ by $m$, then
the total of multiplications/diversions by using the K-Chi\`{o}'s method \eqref{2.2} is about
\[\begin{array}{l}
(k + 1)\left[ {{k^2} + {{(2k)}^2} +  \ldots  + {{(n - k)}^2}} \right] + km\left[ {k + 2k +  \ldots  + (n - k)} \right] + \frac{{n - k}}{k}(m + 1) + m\\
 \approx O\left( {\frac{1}{3}(1 + \frac{1}{k}){n^3}} \right).
\end{array}
\]
Similarly, the computational complexity of other algorithms is also described as follows, see Table 1.

\begin{table}[h]
\label{t1}
\caption{\small Comparisons of the computational complexity for different algorithms on determinant calculations in the nonparallel setting.}
   \leavevmode
    \begin{tabular}{l|c|c} \hline
Algorithms &Multiplications/divisions & Additions/subtractions \\
\hline
Gaussian Elimination (\cite{Golub1996})&$O\left( {\frac{1}{3}{n^3}} \right)$& $O\left( {\frac{1}{3}{n^3}} \right)$ \\
Chi\`{o}'s condensation method (\cite{A1944})&$O\left( {\frac{2}{3}{n^3}} \right)$& $O\left( {\frac{1}{3}{n^3}} \right)$ \\
K-Chi\`{o}'s condensation method(\cite{LAA2014})& $O\left( {\frac{1}{3}(1 + \frac{1}{k}){n^3}} \right)$& $O\left( {\frac{1}{3}{n^3}} \right)$ \\
\hline
    \end{tabular}
\end{table}
From Table 1, we note that additions/subtractions on these algorithms are almost the same.
However, multiplications/diversions mainly depend on the parameter $k$ for K-Chi\`{o}'s condensation method.
But this does not show that the total computational complexity on K-Chi\`{o}'s method \eqref{2.2} is tending to decrease with the $k$ increasing,
since the core loop of the K-Chi\`{o}'s condensation method involves the calculation of $k\times k$ determinants for each element of the matrix during condensation.
Normally, this would necessitate the standard computational workload to calculate the $k$-order determinant, i.e., ${\frac{1}{3}{k^3}}$ multiplications/divisions and
${\frac{1}{3}{k^3}}$ additions/subtractions, using a method such as Gaussian elimination \cite{ken2012}. Therefore, the parameter $k$ is not
the better for the bigger number, see the following experimental results Figure 1 and 2 on the $5000$-order and $20000$-order determinants, respectively.
The small subgraphs in Fig. 1 and 2 show the optimal parameter $k$ value ranges. For example, the optimal parameter $k$ is approximately ten for a $20000$-order determinant.
In addition, for matrices of different dimensions, we specifically compute the optimal parameters $k$, we find the optimal parameter values increasing as the matrix dimension increases.
But this increase is still relatively slow, see Fig. 3.

\begin{figure}[htbp]\label{fig1}
\centerline{\includegraphics[width=7.40in,height=6.92in]{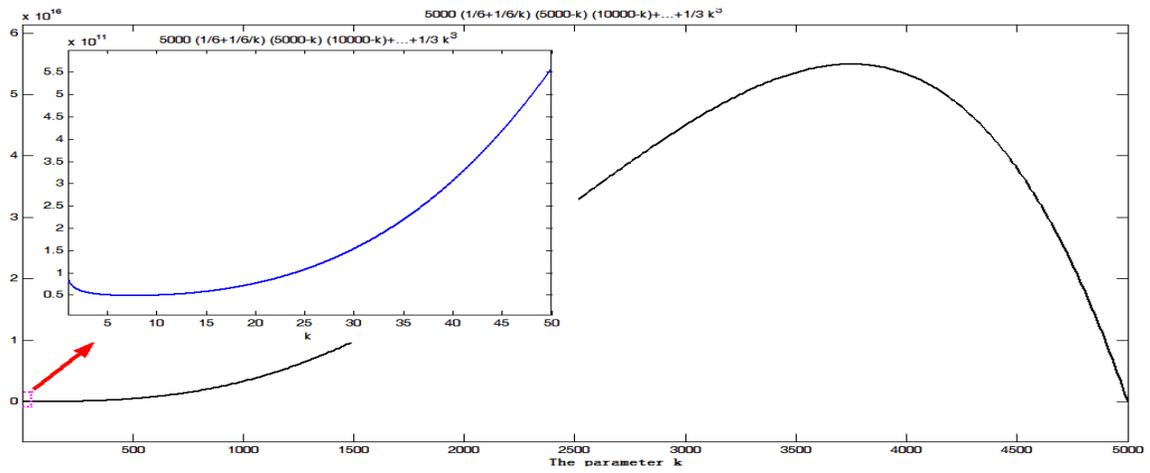}}
\caption{\small The relationship between the parameter $k$ and the number of multiplication/diversion for a 5000-order determinant on K-Chi\`{o}'s condensation method.}
\end{figure}

\begin{figure}[htbp]\label{fig2}
\centerline{\includegraphics[width=7.80in,height=6.92in]{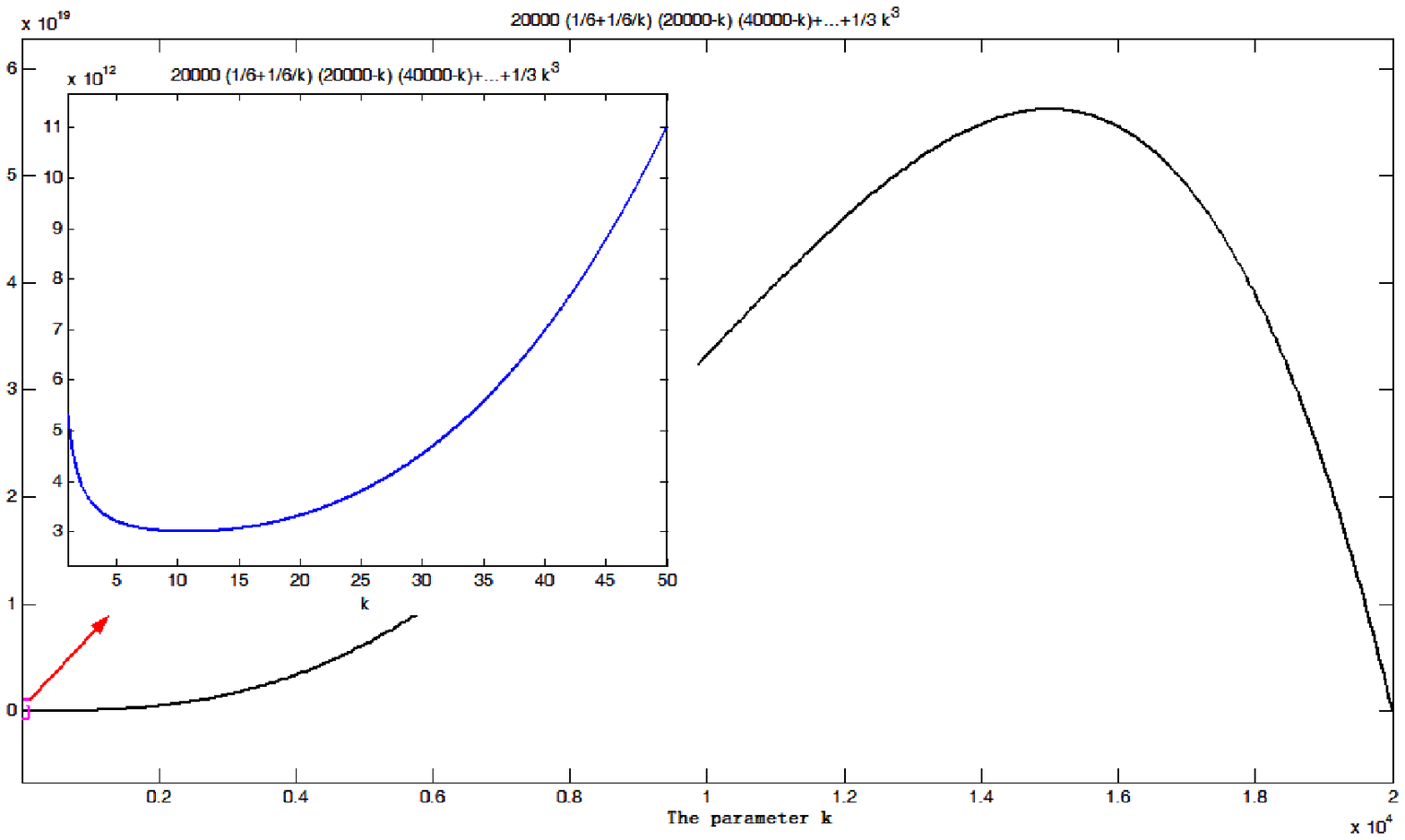}}
\caption{\small The relationship between the parameter $k$ and the number of multiplication/diversion for a 20000-order determinant on  K-Chi\`{o}'s condensation method.}
\end{figure}

\begin{figure}[htbp]\label{fig3}
\centerline{\includegraphics[width=8.00in,height=8.00in]{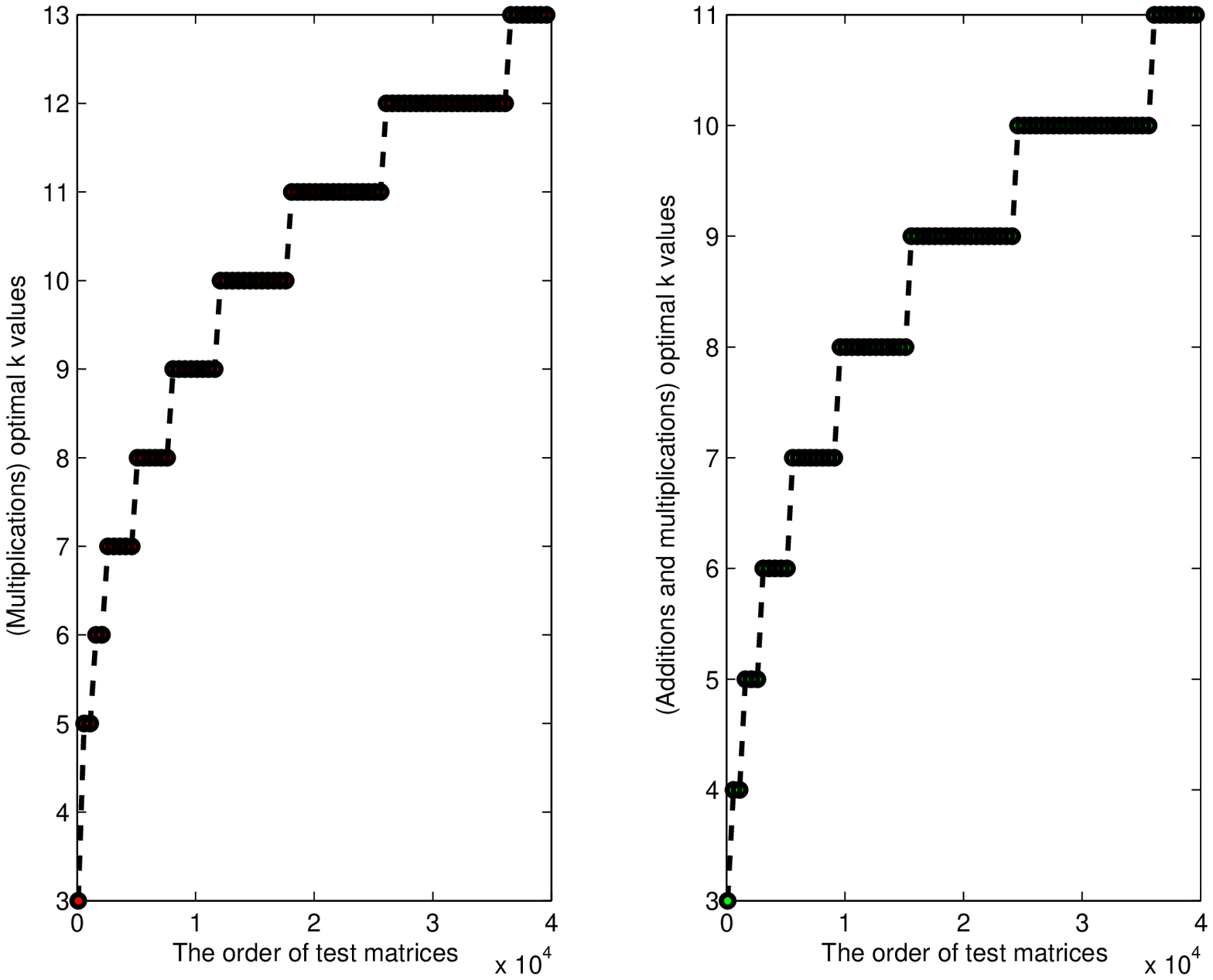}}
\caption{\small The relationship between the optimal parameter $k$ and the dimension of matrices, Left: only consider multiplications; Right: consider all computational complexity.}
\end{figure}

Since the optimal parameter $k$ is usually small, by \eqref{3.30}, we may normalize the each row of matrix ${\hat C}$ by dividing the determinant of $A_0$,
which will further reduce the computational complexity of K-Chi\`{o}'s condensation method, see Example \ref{E4.1}.

\section {An application in the Cramer's rule}

As is well-known, the classical Cramer's rule states that the components of the solution to a linear system in the form
$Ax=b$ (where $A=(a_{ij})$ is an $n\times n$ invertible coefficient matrix) are given by
\begin{equation}\label{3.20}
x_i=\mathrm{det}(A_i(b))/\mathrm{det}(A),\; i=1,2,\ldots,n,
\end{equation}
where $x_i$ is the $i$th unknown.

In \cite{ken2012}, an algorithm based on Chi\`{o}'s condensation and {C}ramer's rule for
solving large-scale linear systems is achieved by constructing a binary, tree-based data flow in which the algorithm mirrors the matrix
at critical points during the condensation process. However, according to the above corollary \ref{Th3.4}, one may obtain certain unknowns values by \textbf{freely controlling the elements in the set} $J'$
without matrix mirroring, see Example \ref{E4.1}. This also makes it more easily for more CPUs to be used in computing process and even
without any communication. At the same time, the scheme \eqref{3.30} also reduce the memory space.

\begin{example}\label{E4.1} Solve the following six-order linear system
\begin{equation}\label{14}
\left[ {\begin{array}{*{20}{c}}
1&3&5&7&9&{11}\\
2&0&0&0&0&9\\
3&0&5&7&0&7\\
4&0&6&8&0&5\\
5&0&0&0&0&3\\
6&5&4&3&2&1
\end{array}} \right]\left[ {\begin{array}{*{20}{c}}
{{x_1}}\\
{{x_2}}\\
{{x_3}}\\
{{x_4}}\\
{{x_5}}\\
{{x_6}}
\end{array}} \right] = \left[ {\begin{array}{*{20}{c}}
\textcolor[rgb]{0,0,1}{1}\\
\textcolor[rgb]{0,0,1}{-1}\\
\textcolor[rgb]{0,0,1}{1}\\
\textcolor[rgb]{0,0,1}{ - 1}\\
\textcolor[rgb]{0,0,1}1\\
\textcolor[rgb]{0,0,1}{ - 1}
\end{array}} \right].
\end{equation}

Let ${I_1} = (1,2,3)$ and ${J_1} = (1,3,5)$, then $J'_1 = (2,4,6)$. Denote
\[|{\mathbf{A_0}}| = \left| {A\left[ {\begin{array}{*{20}{c}}
{{I_1}}\\
{{J_1}}
\end{array}} \right]} \right| = \left| {\begin{array}{*{20}{c}}
1&5&9\\
2&0&0\\
3&5&0
\end{array}} \right|.\]
Then, $|{\mathbf{A_0}}| =90$ and
\[\begin{array}{lll}
{\alpha _4}  &=& \left[ {\begin{array}{*{20}{c}}
{ - \left| {\begin{array}{*{20}{c}}
\mathbf{4}&\mathbf{6}&\mathbf{0}\\
2&0&0\\
3&5&0
\end{array}} \right|,}&{ - \left| {\begin{array}{*{20}{c}}
1&5&9\\
\mathbf{4}&\mathbf{6}&\mathbf{0}\\
3&5&0
\end{array}} \right|,}&{ - \left| {\begin{array}{*{20}{c}}
1&5&9\\
2&0&0\\
\mathbf{4}&\mathbf{6}&\mathbf{0}
\end{array}} \right|,}&{\mathbf{\left| {{A_0}} \right|}}
\end{array}} \right]\left[ {\begin{array}{*{20}{c}}
3&7&{11}&\textcolor[rgb]{0,0,1}{1}\\
0&0&9&\textcolor[rgb]{0,0,1}{ - 1}\\
0&7&7&\textcolor[rgb]{0,0,1}{1}\\
\mathbf{0}&\mathbf{8}&\mathbf{5}&{ \textcolor[rgb]{0,0,1} {\textbf{- 1}}}
\end{array}} \right]\\
 &=& \left[ {\begin{array}{*{20}{c}}
{0,}&{-36,}&{-468,}&\textcolor[rgb]{0,0,1}{-180}
\end{array}} \right];
\end{array}\]

\[
\begin{array}{lll}
{\alpha _5} &=& \left[ {\begin{array}{*{20}{c}}
{ - \left| {\begin{array}{*{20}{c}}
\mathbf{5}&\mathbf{0}&\mathbf{0}\\
2&0&0\\
3&5&0
\end{array}} \right|,}&{ - \left| {\begin{array}{*{20}{c}}
1&5&9\\
\mathbf{5}&\mathbf{0}&\mathbf{0}\\
3&5&0
\end{array}} \right|,}&{ - \left| {\begin{array}{*{20}{c}}
1&5&9\\
2&0&0\\
\mathbf{5}&\mathbf{0}&\mathbf{0}
\end{array}} \right|,}&{\left| {{\mathbf{A_0}}} \right|}
\end{array}} \right]\left[ {\begin{array}{*{20}{c}}
3&7&{11}&\textcolor[rgb]{0,0,1}{1}\\
0&0&9&\textcolor[rgb]{0,0,1}{ - 1}\\
0&7&7&\textcolor[rgb]{0,0,1}{1}\\
\mathbf{0}&\mathbf{0}&\mathbf{3}&\textcolor[rgb]{0,0,1}{\textbf{1}}
\end{array}} \right]\\
 &=& \left[ {\begin{array}{*{20}{c}}
{0,}&{0,}&{-1755,}&\textcolor[rgb]{0,0,1}{315}
\end{array}} \right];
\end{array}
\]

\[\begin{array}{lll}
{\alpha _6} &=& \left[ {\begin{array}{*{20}{c}}
{ - \left| {\begin{array}{*{20}{c}}
\mathbf{6}&\mathbf{4}&\mathbf{2}\\
2&0&0\\
3&5&0
\end{array}} \right|,}&{ - \left| {\begin{array}{*{20}{c}}
1&5&9\\
\mathbf{6}&\mathbf{4}&\mathbf{2}\\
3&5&0
\end{array}} \right|,}&{ - \left| {\begin{array}{*{20}{c}}
1&5&9\\
2&0&0\\
\mathbf{6}&\mathbf{4}&\mathbf{2}
\end{array}} \right|,}&{\left| {{\mathbf{A_0}}} \right|}
\end{array}} \right]\left[ {\begin{array}{*{20}{c}}
3&7&{11}&\textcolor[rgb]{0,0,1}{1}\\
0&0&9&\textcolor[rgb]{0,0,1}{ - 1}\\
0&7&7&\textcolor[rgb]{0,0,1}{1}\\
\mathbf{5}&\mathbf{3}&\mathbf{1}&{\textcolor[rgb]{0,0,1}{ \textbf{- 1}}}
\end{array}} \right]\\
 &=& \left[ {\begin{array}{*{20}{c}}
{390,}&{-234,}&{-2132,}&\textcolor[rgb]{0,0,1}{20}
\end{array}} \right].
\end{array}\]
Therefore, we need only solve the condensed linear system $\hat C{x^{(k)}_{J'}} = {b^{(k)}_{J'}}$, i.e.,

\begin{equation}\label{15}
\left[ {\begin{array}{*{20}{c}}
0&{ - 36}&{ - 468}\\
0&0&{ - 1755}\\
{390}&{ - 234}&{ - 2132}
\end{array}} \right]\left[ {\begin{array}{*{20}{c}}
{{x_2}}\\
{{x_4}}\\
{{x_6}}
\end{array}} \right] = \left[ {\begin{array}{*{20}{c}}
\textcolor[rgb]{0,0,1}{-180}\\
\textcolor[rgb]{0,0,1}{315}\\
\textcolor[rgb]{0,0,1}{20}
\end{array}} \right].
\end{equation}

By Cramer's rule or Gaussian elimination, the solution of above sub-linear system \eqref{15} is
\[{x_2} = 406/117,\;{x_4} = 22/3,\;{x_6} =  - 7/39,\]
which is also the corresponding solution of original linear system \eqref{14}. Similarly, let ${I_2} = (1,2,3)$ and ${J_2} = (2,4,6)$,
then we may also obtain the solution of the unknown $x_1$, $x_2$ and $x_3$:
\[{x_1} = 4/13,\;{x_3} = -10,\;{x_5} =  - 118/117.\]

In addition, we may continue condense the above $\alpha _4$, $\alpha _5$ and $\alpha _6$. For example, we condense them from right side for
$J'_1 = (2,4,6)$. Without loss of generality, we may let $I_3=(6)$, $J_3=(6)$, then $J'_3 = (2,4)$ and we have
\begin{equation}\label{16}
\begin{array}{l}
\alpha _4' = \left[ {468, - 2132} \right]\left[ {\begin{array}{*{20}{c}}
{390}&{ - 234}&\textcolor[rgb]{0,0,1}{20}\\
0&{ - 36}&\textcolor[rgb]{0,0,1}{-180}
\end{array}} \right] = \left[ {182520, - 32760,\textcolor[rgb]{0,0,1}{393120}} \right];\\
\alpha _5' = \left[ {1755, - 2132} \right]\left[ {\begin{array}{*{20}{c}}
{390}&{ - 234}&\textcolor[rgb]{0,0,1}{20}\\
0&0&\textcolor[rgb]{0,0,1}{315}
\end{array}} \right] = \left[ {684450, - 410670,\textcolor[rgb]{0,0,1}{- 636480}} \right].
\end{array}
\end{equation}
By Gaussian elimination, we obtain the solution of the above linear system \eqref{16}:
\[{x_2} = 406/117,\;{x_4} = 22/3.\]

Moreover, to further reduce the computational complexity of K-Chi\`{o}'s condensation method, we may normalize the each row of
matrix ${\hat C}$ by dividing the determinant of $A_0$. For example, the above $\alpha _4'$ and $\alpha _5'$ may be written as
\[\begin{array}{l}
\alpha _4' = \left[ {468/-2132,\textbf{1}} \right]\left[ {\begin{array}{*{20}{c}}
{390}&{-234}&\textcolor[rgb]{0,0,1}{20}\\
0&{ - 36}&\textcolor[rgb]{0,0,1}{ - 180}
\end{array}} \right] = \left[ {-3510/41,630/41, \textcolor[rgb]{0,0,1}{-7560/41}} \right];\\
\alpha _5' = \left[ {1755/ - 2132,\textbf{1}} \right]\left[ {\begin{array}{*{20}{c}}
{390}&{-234}&\textcolor[rgb]{0,0,1}{20}\\
0&0&\textcolor[rgb]{0,0,1}{315}
\end{array}} \right] = \left[ {-26325/82,15795/82,\textcolor[rgb]{0,0,1}{12240/41}} \right].
\end{array}\]

\end{example}

From the above example, we know that applying Gaussian elimination method instead of Cramer's rule to solve the small sub-linear system $\hat C{x^{(k)}_{J'}} = {b^{(k)}_{J'}}$
is also very convenient.

\section {Concluding remarks}

From the above discussion, one can see that unique utilization of matrix condensation techniques yields an elegant process that has promise for
parallel computing architectures. Moreover, as was also mentioned in \cite{ken2012}, these condensation methods become extremely interesting,
since they still retain an $O\left(n^3\right)$ complexity with pragmatic forward and backward stability properties when they are applied to solve
large-scale linear systems by the Cramer's rule or Gaussian elimination.

In this paper, some condensation methods are introduced and some existing problems on these techniques are also
discussed. Though the condensation process removes information associated with discarded columns, this makes the computation of linear systems become
feasible by more freely parallel process.

\textbf{Acknowledgements.} \emph{Partial results of this paper were completed while the first author was visiting the College of William and Mary in 2013. The first author is very grateful to Professor Chi-Kwong Li for the invitation to the College of William and Mary.}



\end{document}